\documentclass[reqno]{amsart}
\usepackage{amsmath}
\usepackage{amssymb}
\usepackage{mathdots}
\newtheorem{thm}{Theorem}[section]
\newtheorem{lemma}{Lemma}[section]
\newtheorem{prop}{Proposition}[section]
\newtheorem{cor}{Corollary}[section]

\theoremstyle{definition}

\theoremstyle{remark}

\def \mcv {\mathcal {V}}
\def \intm {\stackrel{\circ}{M}}
\def \intx {\stackrel{\circ}{X}}

\def \mrn {{\mathbb R}^n}

\def \mrp {{\mathbb R}^p}
\def \mr {{\mathbb R}}

\def \mcs {{\mathcal S}}
\def \mcp {{\mathcal P}}
\def \mcd {{\mathcal D}}
\def \mcq {{\mathcal Q}}
\def \mca {{\mathcal A}}

\def \mcr {{\mathcal R}}
\def \mcp {{\mathcal P}}
\def \mco {{\mathcal O}}
\def \mcv {{\mathcal V}}

\def \mc {{\mathbb C}}
\def \mh {{\mathbb H}}
\def \mn {{\mathbb N}}
\def \mz {{\mathbb Z}}

\def \alf {\alpha}
\def \eps {\epsilon}

\def \del {\delta}

\newcommand{\Id}{\operatorname{Id}}

\newcommand{\Res}{\operatorname{Res}}

\newcommand{\diff}{\operatorname{Diff}}
\newcommand{\Tr}{\operatorname{Tr}}
\newcommand{\I}{\operatorname{I}}
\newcommand{\Ric}{\operatorname{Ric}}

\def \p {\partial}

\numberwithin{equation}{section}
\title[Asymptotically hyperbolic manifolds with polyhomogeneous metric] {Asymptotically hyperbolic manifolds with polyhomogeneous metric}
\author[Leonardo Marazzi]{ Leonardo Marazzi}
\address{Department of Mathematics
Purdue University,
West Lafayette IN 47907, U.S.A.}
\email{lmarazzi@math.purdue.edu}
\begin{document}
\input epsf\
\begin{abstract}
We analyze the resolvent and define the scattering matrix  for asymptotically hyperbolic manifolds with metrics which have a polyhomogeneous expansion near the boundary, and also prove that there is always an essential singularity of the resolvent in this setting.  We use this analysis to prove an inverse result  for conformally compact  odd dimensional Einstein manifolds. 
\end{abstract}
\maketitle

\section{\bf Introduction}\label{intro}

    The objective of geometric scattering theory is to study scattering on certain complete manifolds which are regular at infinity. In this framework one assumes that a complete Riemannian  manifold  can be compactified into a  $C^\infty$ manifold  $X$ with boundary $\p X,$ and the metric, which is necessarily singular at the boundary of the compactified manifold, has a precise asymptotic expansion there which
is modeled in a rather weak sense by well known examples, e.g.  Euclidean, hyperbolic, complex hyperbolic, cylindrical ends, etc.  In this paper we study asymptotically hyperbolic manifolds whose metrics have a polyhomogeneous expansion at the boundary.

Let $X$ be a  compact $C^\infty$ manifold with boundary $\p X$ of dimension $n+1.$ Let $x$ be a boundary defining function on $X,$ which means that $x\geq 0,$ $\p X=\{x=0\},$ and $dx\neq 0$ on $\p X.$ We say that the  compact manifold $X$ is asymptotically hyperbolic with a polyhomogeneous expansion at $\p X$ if there exits a collar neighborhood $U$ of $\p X$ and a diffeomorphism $\Psi: [0,\eps) \times \p X \longrightarrow U$ such that
\begin{gather}\label{hmet}
\Psi^* g= \frac{dx^2 + h(x)}{x^2},
\end{gather}
where $h(x),$ $x\in[0,\eps),$  is a family of metrics on $\p X$ which has an expansion
\begin{gather}\label{hpoly}
h(x,y,dy)\sim h_0(y,dy) + \sum_{ i>0 } x^i\sum_{0\leq j\leq U_i} (\ln x)^j h_{ij}(y,dy).
\end{gather}
Where $U_i\in \mn,$ and $h_0$ is a Riemannian metric on $\p X$ and $h_{ij}$ are symmetric 2-tensors at $\p X.$

We analyze the resolvent  
\begin{gather*}
\mcr(\zeta)= \left( \Delta_g-\zeta(\zeta-n)\right)^{-1}.
\end{gather*}
By the spectral theorem $\mcr(\zeta)$ is bounded in $L^2(X,g)$ provided that $\Re \zeta >>0.$

We denote by 
$$D=\{\zeta\in \mc: \mcr(\zeta) \mbox{ has a pole}, i=1,2\},
$$
and by 
\begin{equation*}
\Gamma= \left\{ \zeta \in \mc:  \zeta \in
\frac{n-\mn}{2} \right\}.
\end{equation*}
We use  techniques of Mazzeo-Melrose \cite{MM}, and Borthwick \cite{Borth} to construct a parametrix for $\mcr(\zeta)$  and use it to show that it can be  continued meromorphically to $\mc\backslash (\Gamma\cup D).$ We prove
\begin{thm}\label{resolvent}
The resolvent:
\begin{equation*}
\mcr_\zeta=\left[\Delta_g+\zeta(n-\zeta) \right]^{-1}:
\dot C^\infty(X)\rightarrow  C^\infty(\stackrel o X)
\end{equation*}
has a meromorphic continuation with respect to $\zeta$ to $\mc \backslash(\Gamma\cup D)$, and
\begin{equation*}
\mcr_\zeta\in\, ^0\Psi^{-2}+\,^0\Psi^{\zeta,\zeta}
+\Psi^{\zeta,\zeta}
\end{equation*}
\end{thm}
Where $^0\Psi^{-2},$ $^0\Psi^{\zeta,\zeta},$
and $\Psi^{\zeta,\zeta}$ are defined in section \ref{sec2}. Using the methods of the proof of theorem \ref{resolvent} and  techniques of \cite{Borth,jsb0,MM} we prove the following 
\begin{thm}\label{new1} Let $x$ be such that \eqref{hmet} is satisfied. Then for  $\zeta\in \mc\setminus(\Gamma\cup D),$  given $f\in C^\infty(\p X),$ there exists a unique $u\in \mca_0(X)$ such that near $\p X,$
\begin{gather}\begin{gathered}\label{theopolyscmat}
\left( \Delta_g+\zeta(\zeta-n)\right) u(x,y)=0;\\
u(x,y)=x^{n-\zeta}F(x,y)+x^\zeta G(x,y),
\end{gathered}\end{gather}
where $F,G \in \mca_0(X),$ $F=f$ at $\p X.$
\end{thm}
The spaces $\mca_0$ are  defined  in section \ref{secpoly} below.
Hence we can define the Poisson operator by

\begin{gather}
\begin{gathered}
E_\zeta: C^\infty(\p X) \longrightarrow \mca_0(X) \\
E_\zeta:f\mapsto u,
\end{gathered}\label{poisson}
\end{gather}
and the scattering matrix $S(\zeta)$   by

\begin{gather}\begin{gathered}\label{defsc}
S(\zeta): C^\infty(\p X) \longrightarrow C^\infty(\p X) \\
S(\zeta):f\mapsto G\mid_{\p X}.
 \end{gathered} \end{gather}

 We prove that the scattering matrix  $S(\zeta)$ at energy $\zeta$ is a pseudodifferential operator of order $2\zeta-n,$and its principal symbol  given by
\begin{gather*}
\sigma_{2\zeta-n}(S(\zeta))(\eta)= 2^{n-2\zeta}\frac{\Gamma(n/2-\zeta)}{\Gamma(\zeta-n/2)}
|\eta|^{2\zeta-n}_{h_0}.
\end{gather*}

 We show that it has a meromorphic continuation to $\mc\backslash (\Gamma\cup D).$ We also study the singularities of the  Resolvent $\mcr(\zeta),$ for a polyhomogeneous metric of the form \eqref{hmet}. We show that for a metric of the form \eqref{hmet} with $h$ given by
\begin{multline}
h=h_0+h_2 x^2+ (\hbox{even powers}) + h_{k,m_k} x^k (\ln x)^{m_k} +\cdots\\
+ h_{k,1} x^k (\ln x) +h_k x^k + O(x^{k+1}(\ln x)^{m_{k+1}}),
\end{multline}
there are higher order poles and the higher order residue at $\zeta=n/2+(k+1)/2$ given by 
$$
\Pi_{n/2+(k+1)/2}-l^{m_l}k\frac{(n-k)/2-1/2}{4}\Tr(h_0h_{k,m_k}),
$$
with $\Pi_{n/2+(k+1)/2}$ a finite rank operator which is the higher order residues of the resolvent. 

We extend the techniques of Guillarmou \cite{Guil} and Graham-Zworski \cite{GZ} to the case of a polyhomogeneous metric and obtain as a corollary

 \begin{cor}\label{corres}
Let $(X,g)$ be asymptotically hyperbolic manifold with 
metric $g$ of the form \eqref{hmet} with $h$ of the form
\begin{equation*}
h=h_0+h_2 x^2+ (\hbox{even powers}) + h_{k,m_k} x^k (\ln x)^{m_k} +\cdots+ h_{k,1} x^k (\ln x) +h_k x^k + O(x^{k+1}(\ln x)^{m_{k+1}}).
\end{equation*} 
If $k\neq n-1,$ and  $m_k>1,$ and the set $\{ \Tr(h_0^{-1}h_{k,m_k})=0 \}$ has zero measure, then $n/2-k/2-1/2$ is a essential singularity 
of $\mcr(\zeta).$   
\end{cor}

Next we study the inverse problem of recovering information about the manifold from the scattering matrix.  We use the methods of \cite{jsb0} to prove

\begin{thm}\label{invpoly}
Let $X,$ $\p X,$ $g_i,$ $S_i,$ for $i=1,2,$  be as above, and let $p\in\p X.$  Then there exists a discrete set $Q\in\mc$ such that if $\zeta\in\mc\backslash Q,$ and $S_1-S_2\in\,^0\Psi^{(2\Re\zeta-n-k;m)}(\p X),$\, $k,m\in \mn_0,$ $k\geq 1,$  near $p.$ Then there exists a diffeomorphism $\psi$ of a neighborhood $U\subset X,$ of $p,$ fixing $\p X,$ such that $\psi^*g_1-g_2=O(x^{k-2}(\ln x)^m).$ 
\end{thm}
The spaces $^0\Psi^{(2\Re\zeta-n-k;m)}$ correspond to the natural analog to this case of the class of $\mcv_0$ pseudodifferential operator introduced for instance in \cite{MM}. We define these spaces in section \ref{sec2}. 
We apply theorem \ref{invpoly} to solve an inverse problem in odd dimensional Einstein manifolds.  We say that an asymptotically hyperbolic manifold 
$(X,g)$ of dimension $n+1$ is Einstein  if $g$ satisfies the condition
$$
\Ric (g)=-ng,
$$ 
where $\Ric$ is the Ricci curvatuture tensor.
Einstein manifolds have been studied by C.R.Graham \cite{Grah1}, 
C.R. Graham and M.Zworski \cite{GZ}, P.T. Chru\'sciel, et al. 
\cite{PDLS}, C. Guillarmou and A. S\'a Barreto \cite{GZ}, among 
others.   

When dim X$=n+1$ is even and $X$ is an Einstein manifold the tensor $h(x)$ defined in \ref{hmet} is $C^\infty$ up to $\p X,$ and the scattering matrix $S(\zeta)$ is well defined at $\zeta=n,$ the following inverse theorem was proved in  \cite{GSa}
\begin{thm}\cite{GSa}
Let $(X_i,g_i)$ $i=1,2,$ be $n+1$ even dimensional conformally compact Einstein manifolds, then if the scattering map 
$$
S_1(n)\mid_\mco=S_2(n)\mid_\mco,
$$  
where $S_i$ is the scattering matrix on $X_i,$ $i=1,2.$ The set  $\emptyset\neq\mco\subset\p X_1 \cap \p X_2$ is an open set, and $\Id: \mco\subset\p X_1\mapsto\p X_2$ is a diffeomorphism. 
Then there is a diffeomorphism 
$$
J:\bar X_1\rightarrow \bar X_2,
$$ 
such that $J^*g_2=g_1.$ 
\end{thm}
However unlike when $n+1$ is even, when $n+1$ is odd the tensor $h$ has a polyhomogeneous asymptotic behavior near the boundary $\p X,$ and that is the motivation for  the study of scattering on manifolds having polyhomogeneous metrics. In this case the scattering matrix has a pole at $\zeta=n,$ since $\Gamma(-n/2)$ has a simple pole for $n/2$ a positive integer. %The Poisson operator $\mcp_\zeta,$ on the other hand, is meromorphic at $\zeta=n,$  since $0\notin \sigma_{pp}(\Delta_g),$ where $\sigma_{pp}(\Delta_g)$ is the point spectrum of the Laplacian. We can apply Propositions 3.5 of \cite{GZ}.  
We study the principal  symbol of  the residues of the scattering matrix and the modified scattering operator (MSO)
\begin{equation}
\tilde\mcs f=\frac{d[(n-\zeta)S(\zeta)]}{d\zeta}\mid_{\zeta=n},
\end{equation}
to get an inverse result in this case. But it turns out it suffices to consider the MSO to obtain an inverse result. 

  It was proven by Graham \cite{Grah1} that if dim $X=n+1$ is odd and $(X,g)$ is asymptotically hyperbolic and Einstein then the family $h(x)$ defined in \eqref{hmet} is of the form
\begin{gather*}
h(x)=h_0(y,dy)+(\mbox{even powers}) + h_n x^n \ln x + F_n x^n + \cdots. 
\end{gather*}
C.Fefferman and C.R.Graham \cite{FG} proved that the coefficients $h_0$ and $F_n$ determine the entire expansion of the metric at $x=0,$ we use a unique continuation theorem of Biquard \cite{biq}\footnote{The approach in \cite{GSa} can also be used to prove the theorem.} (we state it in theorem \ref{uniquecont}) which uses results of Fefferman and Graham to obtain an isometry on a neighborhood of the boundary. We use theorem 4.1 of Lassas-Taylar-Uhlmann\cite{LTU} to extend this isometry to the whole manifold and prove

 \begin{thm}\label{inveinstein}
Let $X_i,$ $\p X_i,$ $g_i,$ for $i=1,2,$ be  $n+1$-dimensional Einstein manifolds; and let $S_i$ for $i=1,2,$ be the corresponding scattering  matrix. Assume $\emptyset\neq\mco\subset\p X_1 \cap \p X_2$ an open set, and that $\Id: \mco\subset\p X_1\mapsto\p X_2$ is a diffeomorphism.  If 
$$
\tilde\mcs_1 f\mid_\mco=\tilde\mcs_2 f\mid_\mco
$$
for all $f\in\mc^\infty_0(\mco).$ 
Then there exists a diffeomorphism $\psi$ satisfying $\psi^* g_2=g_1.$ 
\end{thm}

$Acknowledgments:$ The author thanks C. Guillarmou and A. S\'a Barreto for many helpful discussions, and C. Guillarmou for pointing out a mistake on an earlier version.

\section{\bf Laplacian}\label{seclaplacian}

We assume the metric $h$  has a "polyhomogeneous" expansion in $x$ of the form
\begin{gather}\label{hpoly0}
h(x,y,dx,dy)=\sum_{i,j}\left[ h_{ij}(0,y,dy)+\sum_{m\in\mn} x^{k_{ij}^{(m)}} \sum_{0\leq l_{ij}^{(r)}\leq U_{ijm}}(\ln x)^{l_{ij}^{(r)}} \tilde h_{ij}^{(m,r)}(y)\right]dy_idy_j+\sum_j h_j(x,y)dxdy_j,
\end{gather}
where $l_{ij}^{(r)}\geq 0$ and $k_{ij}^{(m)}> 0$ for every $i,j,m,r;$ and $h_j$ is polyhomogenous in $x,$ and  $h\mid_{x=0}$ induces a Riemannian metric on $\p X$.

We prove that there exists a diffeomorphism that  puts the metric to normal form and the resulting metric is again polyhomogeneous, as stated in the introduction,
\begin{lemma}
 Let $X$ be a smooth manifold with boundary $\p X$  with a metric $g$ of the form \eqref{hmet}, with $h$ as in \eqref{hpoly0} in some product decomposition near $\p X$, $x$ being the defining function for $\p X$, and such that $h\mid_{x=0}$ is independent of $dx^2$.
Then fixed $h_0$ there exists a unique $x$ such that \eqref{hmet} holds, with $h$ of the form \eqref{hpoly}.
\end{lemma}
\begin{proof}
The argument of Lemma 2.1 of \cite{Grah1} applies verbatim, we only need to prove that $h$ has the polyhomogeneous expansion stated. The proof in \cite{Grah1} is based on a change of variable, $x'=xe^\omega$ which gives $dx^2+h=e^{2\omega}g_0$, together with the condition $|dx|^2_{dx^2+h}=1$ and for $\omega$ prescribed at the boundary, the PDE
\begin{gather*}
2(\nabla_{g_0}x)(\omega)+x|d\omega|^2_{g_0}=\frac{1-|dx|^2_{g_0}}{x}
\end{gather*}
can be solved to get $\omega,$ since it is a non-characteristics first order PDE. If $g_0$ is a polyhomogeneous metric of the form \eqref{hpoly}, from $dx^2+h=e^{2\omega}g_0$ we have that $\omega$ is necessarily polyhomogeneous.
\end{proof}

A straightforward calculation using geometric series arguments and the definition of the determinant of a metric shows that we can write
\begin{gather}\label{laplaceminuszeta}
P(\zeta)=\Delta_g-\zeta(\zeta-n)=\sum_{j+|\alf|=0}^k p_{j,\alf}(x,y)(xD_x)^j(xD_y)^\alf-\zeta(\zeta-n),
\end{gather}
with $p_{j,\alf}$ a $C^\infty$ function in the interior and  polyhomogeneous in $x$ close  to the boundary. Spaces including operators of this kind on conformally compact manifolds were introduced in \cite{melcor,Borth}, and we will recall them in the following section.

Let $R_r$ be the radial $\mr^+$ action (multiplying by r) on the tangent space, and $f$ be the exponential function, then we can define the normal operator:
\begin{gather*}
N_p(P)u=\lim_{r\rightarrow 0} R_r^*f^*P(f^{-1})^*R_{1/r}^* u.
\end{gather*}
Taking the limit when $r$ goes to zero gives
\begin{gather*}
N_p(P)=\sum_{j+|\alf|=0}^k p_{j,\alf}(0,y)(xD_x)^j(xD_y)^\alf,
\end{gather*}
which expresses $N_p$ as just freezing the coefficients $p_{j,\alf}$ at p.

\section{\bf Polyhomogeneous conormal distributions}\label{secpoly}
We recall the spaces of functions introduced in \cite{melcor}. Let $M$ be a smooth manifold with corners as defined in \cite{melcor}, and
let $\rho=(\rho_1,...,\rho_p)$ be the defining functions for the finitely many
boundary faces
$Y_1,...,Y_P$ of $M$. Let $\mcv_b(M)$ be the set of smooth vector fields
tangent to the boundary, let $\beta \in C^\infty(M;\mrp),$ and
$m=(m_1,...,m_p)\in \mrn$ a multi-index. We recall the space of conormal distributions
$$\mca^{m}(M)=\{u\in C^\infty(\intm):\mcv_b^ku\in\rho^mL^\infty(M),  \forall k\},$$ $$\mca^{m-}=\bigcap_{m'<m}\mca^{m'},$$ and $$\mca_{\beta}(M)=\left\{u\in C^\infty(\intm ):\left[\prod_{l=0}^p\prod_{k=0}^{m_l-1}
(\rho_j\p_{\rho_j}-k)^{k+1} \right](\rho^{-\beta}u)
\in \mca^{m-}(M) \right\}.$$

We generalize this definition to allow leading terms having logarithmic functions, for $\beta$  and $\alf\in \mc^\infty(M;\mrp)$ we define the generalized space of
polyhomogeneous distributions:

\begin{equation}
\mca_{\beta;\alf}(M)=\left\{u\in C^\infty(\intm):\left[\prod_{l=0}^p\prod_{k=0}^{m_l-1}
(\rho_j\p_{\rho_j}-k)^{k+1} \right](\rho^{-\beta}(\ln \rho)^{-\alf}u)
\in \mca^{m-}(M) \right\}.
\end{equation}
Where $T_j=\rho_j\p_{\rho_j}$.
 It was proven in \cite{Borth} that for $$u\in C^\infty(\intm),$$
\begin{gather*}
u\in\mca_\beta(M)\Leftrightarrow u\sim \sum_{0\leq l\leq k<\infty} \rho_j^{\beta_j+k}(\ln \rho_j)^l a_{k,l},
\end{gather*}
 at each boundary surface $Y_j$, where "$\sim$" means that there is an asymptotic expansion of the given form.  Here $a_{k,l}\in\mca_{\beta^{(j)}}(Y_j)$, $\beta^{(j)}$ is a multi-index on each face $Y_j$ associated to $\beta$, looking at $Y_j$ as a manifold with boundary itself and setting $\beta^{(j)}_i=\beta_m\mid_{H_i},$ for $H_i$ a boundary hypersurface of $Y_j,$ and $Y_m$ the unique other boundary surface such that $H_i$ is a component (the corner between) of $Y_j\cap Y_m$. The proof given there is by constructing an expansion for $\rho^{-\beta}u$, and thus it suffices to prove it for $\beta=0.$ The same proof holds if we get the expansion for $\rho^{-\beta}(\ln \rho)^{-\alf}u,$ therefore we can assume $\beta=\alf=0.$ We have
 \begin{thm}
 For $u\in C^\infty(\intm),$ $u\in\mca_{\beta;\alf}(M)$ if and only if $u$ satisfies
 \begin{gather}\label{polyh}
u\sim \sum_{0\leq l\leq k<\infty} \rho_j^{\beta_j+k}(\ln \rho_j)^{\alf_j+l} a_{k,l},
\end{gather}
 at each boundary surface $Y_j$, where  $a_{k,l}\in\mca_{\beta^{(j)}}(Y_j)$, $\beta^{(j)}$ and $\alf^{(j)}$ are multi-indexes on each face $Y_j$ associated to $\beta$ and $\alf,$ respectively.
 \end{thm}

\section{\bf Stretched product}

The type of manifold with corners we need here is  obtained by blowing
up the product $X\times X$ along $\p \Delta \iota$, where $\p \Delta \iota=
(\p X \times \p X)\cap \Delta \iota \cong \p X$, and $\Delta \iota$ is the set of fixed
points of the involution $I$ that exchanges the two projections; $I$ satisfies
\begin{gather*}
I(\pi_L(X\times X))=\pi_r(X\times X),
\end{gather*}
where $\pi_L(X\times X)$ is the projection onto the first component, and
$\pi_r(X\times X)$ the projection onto the second component.

We use the usual notation for the stretched product $X\times_0 X$ and denote
the blow-down map by

\begin{gather}\label{b}
b:X\times_0 X\rightarrow X\times X.
\end{gather}

To analyze the functions using the blow-down $b$ map we look at the pull-back of the function under $b.$ This process is known as the blow-up of the manifold $X\times X$ and amounts to the introduction of singular
coordinates near the corner. We can use a different subset of these coordinates near each face,  near the left face, in local projective coordinates,  we use (with $Y=y-y'$)
\begin{gather}\label{lface}
s=\frac{x}{x'},\quad z=\frac{Y}{x'},\quad x',\quad y',
\end{gather}
near the front face we use
\begin{gather}\label{fface}
\rho=\frac{x}{|Y|},\quad \rho'=\frac{x'}{|Y|},\quad r=|Y|,
\quad \omega=\frac{Y}{|Y|},\quad y,
\end{gather}
near the right face  we use
\begin{gather}\label{rface}
t=\frac{x'}{x},\quad z'=-\frac{Y}{x},\quad x,\quad y.
\end{gather}

Setting
\begin{gather*}
R=\sqrt{(x')^2+x^2+|y-y'|^2}
\end{gather*}
the left, right, and front faces are characterize by $\rho=0$, $\rho'=0$, and
$R=0$ respectively.

We can also define other blow up that will be useful for the process of defining the scattering matrix through operators having Schwartz kernels whose pull-backs can be computed explicitly (e.g. equation \eqref{Mscatt}). Let $X\times_0\p X$ be the manifold with corners obtain by blowing up $X\times \p X$ along the diagonal $\Delta\subset \p X\times\p X$, and
$$
\tilde b: X\times_0\p X\rightarrow X\times \p X,
$$
the corresponding blow-down map, and let
$$
M=\overline{{\tilde b} ^{-1}(\p X\times \p X \backslash\Delta)}.
$$
Then
$$
b_\p=b\mid_M:M\sim \p X\times_0\p X\rightarrow\p X\times\p X,
$$
corresponds to the manifold $\p X\times\p X$  blown-up along the diagonal $\Delta\subset \p X\times \p X.$

\section{\bf $\mcv_0$ Operators}\label{sec2}

In this section we introduce the generalization needed  of the pseudodifferential operators modeled by $\mcv_0(X)$ which were used in \cite{MM}. In what follows $\zeta,\zeta',$  are complex numbers, and in most of the applications (to the parametrix construction) $\zeta'=\zeta'=0.$ 
For convenience we first introduce the spaces of half densities of the form
\begin{equation*}
h'=|h(x,y)|\frac{dx}{x}\frac{dy}{x^n},
\end{equation*}
with $h$ as before.
We denote the bundle of singular half-densities by
$\Gamma_0^{1/2}=\Gamma_0^{1/2}(X)$, the canonical section of $\Gamma_0^{1/2}$ is of the form
\begin{equation*}
\nu=|h(x,y)|^{1/2}
\left|\frac{dx}{x}\frac{dy}{x^n}\right|^{1/2}.
\end{equation*}
We consider the continuous linear maps, from the space of smooth sections vanishing to infinite order at the boundary to the space of extendible sections

\begin{equation}\label{B}
B: \dot C^\infty (X;\Gamma_0^{1/2})\rightarrow C^{-\infty}(X;\Gamma_0^{1/2}).
\end{equation}
In the  blow-up $X\times_0 X$ the extension of the
half-density bundle is given by
\begin{equation*}
\Gamma_0^{1/2}=(\pi_l)^*(\Gamma_0^{1/2})\otimes (\pi_r)^*(\Gamma_0^{1/2}),
\end{equation*}
Where $\pi_L(X\times X)$ is the projection onto the first component $X\times \p X$, and
$\pi_r(X\times X)$ the projection onto the second component $\p X\times X$.
This bundle is well defined, the canonical projection
\begin{equation*}
b:X\times_0 X \rightarrow X\times X
\end{equation*}
lifts $\Gamma_0^{1/2}(X\times X)$ to $\Gamma_0^{1/2}(X\times_0 X)$. We also introduce the corresponding class of $\mcv _0$ polyhomogeneous pseudodifferential operators by
\begin{gather*}
B\in \,^0\tilde \Psi^m(X,\Gamma^{1/2}_0) \Leftrightarrow \kappa(B) \in \,^0 \tilde K^m(X),
\end{gather*}
where $\kappa(B)$ is the lift to $X\times_0 X$ of the kernel of the map
$B$ defined in \eqref{B}, and  $^0\tilde K^m(X)$ is the space of polyhomogeneous conormal sections of order m of
the bundle $\Gamma_0 ^{1/2}$ associated to $\Delta \iota_0$ ($=\{ s=1,z=0\},
s=\frac{x}{x'},\quad z=\frac{Y}{x'}$) with coefficients given by  polyhomogeneous distributions and required to
vanish to all orders at the boundary components other than the front face.\\
As in \cite{Borth}, we also define $\,^0\Psi^{(\zeta;\zeta'),(\zeta;\zeta')}(X\times_0 X,\Gamma_0^{1/2})$ to be
the class of operators whose Schwartz kernel satisfy
\begin{gather*}
b^*K\in \mca_{-\infty,(\zeta;\zeta'),(\zeta;\zeta')}(X\times_0 X,\Gamma_0^{1/2}),
\end{gather*}
and are extendible across the front face (so no logarithmic terms there). The residual class of the
construction is $\Psi^{(\zeta;\zeta'),(\zeta;\zeta')}$ the operator with kernels in
$\mca_{(\zeta;\zeta'),(\zeta;\zeta')}(X\times X,\Gamma_0^{1/2})$.

Since the kernel $\kappa(B)$ of an operator $B\in \,^0\tilde\Psi^m(X)$ is polyhomogeneous conormal with respect to the lifted diagonal $\Delta \iota_0$ it can be restricted to a fibre $F_p$ of the front face lying over the point $(p,p)\in\p \Delta \iota=\{x=x'=Y=0\}$; this restriction is called  the normal operator, it was introduced in section \ref{seclaplacian},
\begin{gather*}
N_p(B)=\kappa(B)\mid_{F_p}.
\end{gather*}

The bundle $\Gamma_0 ^{1/2}$ is trivial, thus the normal operator can be defined as a convolution operator:
\begin{gather*}
[Np(B)f](x,y)=\int k(0, y, s,z)f\left(\frac{x}{s},y-\frac{x}{s}z\right)
\frac{ds}{s}dz \cdot \mu
\end{gather*}
where
\begin{gather*}
\quad \mu=|h(x,y)|\left|\frac{dx}{x}\frac{dy}{x^n} \right|
;\qquad f=f(x,y)\mu.
\end{gather*}

The construction for the symbol map $^0 \sigma$ can be carried out as in Mazzeo and Melrose (\cite{MM} section 5), although it needs the corresponding modification to polyhomogeneous operators. We think of the symbol of the kernel $\kappa(B)$ as a symbolic density on the fibres of the polyhomogeneous conormal bundle $\tilde N^*$ of the lifted diagonal,
\begin{gather*}
\sigma_m(\kappa(B))\in S^m(\tilde N^*(\Delta\iota_0); \Gamma_0(X)\otimes\Gamma(\mbox{fibre}))\hspace{3mm} \mbox{mod}\,\,S^{m-1}.
\end{gather*}
There is a natural isomorphism
\begin{gather*}
\delta:\tilde N^*(\Delta\iota_0)\leftrightarrow \,^0\tilde T^*X,
\end{gather*}
of the polyhomogeneous conormal bundle with the bundle $^0\tilde T^*X,$ dual to the bundle $^0\tilde TX$ of which the sections are the elements of $\mcv_0$ generated by $x\p_x, x\p_y,$ with  coefficients that are polyhomogeneous functions. The isomorphism $\delta$ is the dual to
\begin{gather*}
\tilde N(\Delta\iota_0)\leftrightarrow \,^0\tilde TX,
\end{gather*}
the lifting to the diagonal with respect to the blow-up local coordinates. In the blow-up coordinates the polyhomogeneous conormal bundle is spanned by $s\p_s, s\p_z$, and  coefficients that are  polyhomogeneous distributions in these coordinates (the lift of the corresponding quotients on the base space $X\times X$). To define the symbol we divide by the lift $\omega_0$ of the symplectic density form $h'$,
\begin{gather*}
^0\tilde \sigma_m(B)=\delta^*[\tilde \sigma_m(\kappa(B))]/\omega_0\in\tilde S^m(\,^0T^*X)\hspace{3mm} \mbox{mod}\,\,\tilde S^{m-1}.
\end{gather*}
 Such a symbol satisfies an exact sequence just as in \cite{MM}, we state this as a  
\begin{prop}\label{exact}
For any $m\in \mr$ the
symbol map gives a short exact sequence:

\begin{gather*}
0\rightarrow \,^0\tilde \Psi^{m-1}(X) \rightarrow \,^0\tilde \Psi^{m}(X)\rightarrow
\tilde S^m(^0 T^*X)/\tilde S^{m-1}(^0 T^*X)\rightarrow 0
\end{gather*}
such that:
\begin{gather*}
^0 \tilde \sigma_{m+m'}(B\cdot B')=\, ^0 \tilde \sigma_{m}(B)\cdot\, ^0 \tilde \sigma_{m'}( B')\,\,
\mod\,\tilde S^{m+m'-1}(\, ^0T^*X)
\end{gather*}
for
\begin{gather*}
B, B': \dot C^\infty (X;\Gamma_0^{1/2})\rightarrow \dot C^\infty (X;\Gamma_0^{1/2});
\,\,\, B\in \,^0\tilde \Psi^m, B'\in \,^0\tilde \Psi^{m'}.
\end{gather*}
\end{prop}

We can also define
\begin{gather}\label{pseudofaces}
^0\Psi^{p,(a;a'),(b;b')}(X)=\,^0\Psi^{(a;a'),(b;b')}(X)+\,^0\tilde \Psi^p(X).
\end{gather}
If $B\in\,^0\Psi^{\infty,(b,b')}(X)$ we have
\begin{equation*}
B:x^pL^2(X;\Gamma_0^{1/2})\rightarrow C^\infty(X,\Gamma_0^{1/2})\qquad\mbox{for}\quad p>n-b. 
\end{equation*}
This allows the composition with differential operators in $\,^0\tilde\Psi^k(X;\Gamma_0^{1/2})$ for any $k$. This gives composition with the type of operator we will get:
\begin{gather}\label{polycomp}
\diff_0^k(X;\Gamma_0^{1/2})\cdot\,^0\Psi^{m,(a;a'),(b;b')}(X;\Gamma_0^{1/2})\subset\,^0\Psi^{m+k,(a;a'),(b;b')}(X),
\end{gather}
and we have that their symbol is well defined (as a polyhomogeneous symbol) and satisfies an exact sequence as in theorem \ref{exact}. We also have
\begin{equation*}
^0 \tilde \sigma_{m+m'}(P\cdot B)=\, ^0 \tilde \sigma_{m}(P)\cdot\, ^0 \tilde \sigma_{m'}( B)
\qquad\mbox{for}\quad P\in\,^0\tilde\Psi^m(X;\Gamma_0^{1/2}),\quad B\in ^0\tilde\Psi^{m',a,b}(X;\Gamma_0^{1/2}).
\end{equation*}
For the parametrix construction we follow \cite{MM}, the idea is to solve the equation applying the normal operator and iterate the process, this produces a filtration, described in the following 
\begin{prop}\label{exactnormal}
The normal operator defines an exact sequence
\begin{multline*}
0\rightarrow R \,^0\Psi^{(a,a'),(b,b')}(X;\Gamma_0^{1/2}) \rightarrow\,^0\Psi^{(a,a'),(b,b')}(X;\Gamma_0^{1/2}) \\
\stackrel{N}\longrightarrow\mca_{(a,a'),(b,b')}(F;\Gamma_0^{1/2}
(X^l_*)\otimes \Gamma_0^{1/2}(X^r_*))\rightarrow 0.
\end{multline*}
Where $U(h)$ is a constant that depends on the metric $h,$  and for any operator $$P\in\,^0\tilde\Psi^k(X;\Gamma_0^{1/2}),\qquad    B\in\, ^0\Psi^{m',(a,a'),(b,b')}(X;\Gamma_0^{1/2}),$$ we have
\begin{gather*}
N_P(P\cdot B)=N_P(P)\cdot N_P(B).
\end{gather*}
\end{prop}

The second filtration is provided by the indicial operator (defined below), for this case it filters the lower order term in $"x"$ and also the higher order term in $"\ln x",$ as it would be expected in order for the parametrix to work;
the details of the parametrix are discussed on Section \ref{secparametrix}.

\section{\bf Parametrix, the Poisson operator and the scattering matrix}\label{secpoiscmat}

\subsection{\bf The indicial operator}
For $g$ as in
\eqref{hmet} initially one wants to study
\begin{equation}\label{laplambda}
\Delta_g-\lambda^2.
\end{equation}
The zeros of this operator or eigenvalues of the Laplace-Beltrami operator, with absolutely continuous spectrum $\lambda^2\in [n^2/4,\infty).$  In order to make this part of the spectrum to be the positive real numbers, we subtract $n^2/4$ to \eqref{laplambda} and it is standard to consider the parameter $\zeta$ so that 
$$
\lambda^2-n^2/4=-\zeta(n-\zeta),
$$
this can be solve by
\begin{gather}\label{sigma}\begin{gathered}
\zeta(n-\zeta)+\lambda^2-\frac{n^2}{4}=0\\
\Rightarrow \zeta_{\pm}=\frac{n}{2}\pm i\lambda,
\end{gathered}\end{gather}
and from now on we use the parameter $\zeta$ only. The operator looks like
\begin{equation}
\Delta_g-\zeta(\zeta-n).
\end{equation}
Let $I$ the indicial operator
$$
I\left( \Delta_g-\zeta(\zeta-n)\right)=-x^2\p_x^2+(n-1)x\p_x-\zeta(\zeta-n).
$$
The indicial roots are obtained by setting the coefficient of the leading order term (i.e. $x^\eta$)
$$
\I [\Delta_g-\zeta(\zeta-n)]x^\eta
$$
 to be zero. The solutions given by:
$$
\eta=\zeta,\qquad \eta=n-\zeta.
$$
It was proven in \cite{Guil} that there are especial energies, where the resolvent has poles or essential singularities, given by

\begin{equation*}
\Gamma= \left\{ \zeta \in \mc:  \zeta \in
\frac{n-\mn_0 }{2}\right\};
\end{equation*}
we stay away from these points, and the set of scattering poles $D$ defined in the introduction.

\subsection{\bf Parametrix}\label{secparametrix}

As mentioned at the end of Section \ref{sec2}, we are going to keep track of not only the powers of $x$, but those of $\ln x$, to do so the lemma and proposition needed were proved in \cite{Borth}, the generalizations of lemma 3.2 and proposition 4.2 of that paper are:

\begin{lemma}\label{inductparametrix}
Let $\zeta \in \mc \backslash (\Omega\cup D),$ $k\in\mn,$ then for $v \in \mca_{\sigma+1;k}$,
we can find
$u\in\mca_{\sigma+1;k}$ such that
\begin{equation*}
v-[\Delta_g-\zeta(n-\zeta)]u \in \dot C^\infty(X).
\end{equation*}
 \end{lemma}

 We denote the model Laplacian by $\Delta_0=-(x\p_x)^2+nx\p_x-(x\p_y)^2,$ and by $\mcq$ be the hyperbolic space $\mh^{n+1}$ blown-up at a boundary point.

\begin{prop}\label{b42poly}
The model resolvent $\mcr_0(\zeta)=[\Delta_0-\zeta(n-\zeta)]^{-1}$ can be extended to a meromorphic map
$$
\mcr_0(\zeta):\mca_{(\zeta+k;k'),\zeta-l}(Q)\rightarrow \mca_{\zeta,\zeta-l}(Q),\qquad\mbox{for}\quad k,k'\in\mn,\quad l\in\mn_0,
$$
with poles $\zeta\in\frac{1}{2}(n-k-\mn_0)\cup \frac{1}{2}(-l-\mn_0)$ and $-\mn_0$ for $n$ odd.
 \end{prop}

The proof of this proposition is identical to the one of proposition 4.2 in \cite{Borth}. Using this Proposition we provide a proof of the meromorphic continuation of the resolvent:

\begin{prop}\label{parametrixpoly}
Let $\zeta\in \mc \backslash \Gamma$, then there exists $M_\zeta$
analytic, such that:
\begin{equation*}
[\Delta_g-\zeta(n-\zeta)]M_\zeta=I-F_\zeta
 \end{equation*}
with $M_\zeta\in\, ^0\Psi^{-2}+\,^0\Psi^{\zeta,\zeta}$ and
$F_\zeta \in \Psi^{\infty,\zeta}$.
\end{prop}

\begin{proof}
 Let $P(\zeta)$ be as in \eqref{laplaceminuszeta}, the first stage is to find $M_0(\zeta)\in \,^0\tilde\Psi^{-2}(X)$ so that
\begin{equation}
P(\zeta)\cdot M_0(\zeta)-\Id=Q_1(\zeta)\in\,^0\tilde\Psi^{-\infty}(X),
\end{equation}
and can be carried out the same as in the asymptotically hyperbolic case.
The next stage is to construct $M_1(\zeta)\in\,^0 \Psi^{\zeta,\zeta}(X)$ so that
\begin{equation}
P(\zeta)\cdot M_1(\zeta) -Q_1(\zeta)=Q_2(\zeta)\in R^\infty \cdot \,^0 \Psi^{\zeta,\zeta}(X).
\end{equation}
For that we look for $M_{1,0}\in\,^0 \Psi^{\zeta,\zeta}(X)$  such that
\begin{equation}\label{firstpara}
P(\zeta)\cdot M_{1,0}(\zeta)-Q_1=Q_{1,1} 
\in R \cdot\,^0 \Psi^{\zeta,\zeta}(X).
\end{equation}
To find such a term the normal operator comes into use, we solve
\begin{equation}\label{normalpara}
N_p(P)\cdot N_p(M_{1,0}(\zeta))=N_p(Q_1).
\end{equation}
Since $Q_1\in \,^0\Psi^{-\infty}(X)$ the normal operator of $Q_1$ is in $C^\infty$ on the front face and vanishes to infinite order at the boundary. Thus under the identification of the interior of each leaf of the front face of $X\times_0 X$ with $\mcq$, \eqref{normalpara} reduces to
\begin{equation}
[\Delta_0-\zeta(\zeta-n)]N_p(M_{1,0})=N_p(Q_1(\zeta))\in \dot C^\infty.
\end{equation}

By proposition \ref{b42poly} the last equation can be solved meromorphically in $\zeta$; and by surjectivity of the normal operator a solution to \eqref{firstpara} can be found modulo the  remainder $R \cdot Q_{1,1}$. A better remainder can be obtained using   \eqref{polycomp} to compose operators, and with $\kappa$ denoting the kernel of the specific operator we have
\begin{equation}
\kappa(P(\zeta)\cdot M_{1,0}(\zeta))=I(P(\zeta))\cdot \kappa(M_{1,0}(\zeta)),
\end{equation}
modulo a term that vanishes to one order higher i.e. in $\,^0\Psi^{(\zeta+1;\zeta'),\zeta}(X),$ where $\zeta'$ depends on the metric $h$. By the choice of $M_{1,0}\in\,^0\Psi^{\zeta,\zeta}(X),$ and the fact that $\rho^\zeta$ is a solution of the indicial operator modulo higher order we get that
\begin{equation*}
P(\zeta )\cdot M_{1,0}(\zeta)-Q_1=x (\ln\rho)^{\zeta'} L_{1,1},\qquad  L_{1,1}\,\in\,^0 \Psi^{\zeta,\zeta}(X).
\end{equation*}
We look for a series 
$$
M_1(\zeta)\sim\sum_k R^k M'_{i,k}(\zeta),\qquad M'_{i,k}(\zeta)\in\,^0 \Psi^{\zeta,\zeta}(X).
$$ 
For that it is easier to look for a series of the form
$$
M_1(\zeta)\sim\sum_k (x')^k M_{i,k}(\zeta),\qquad M_{i,k}(\zeta)\in\,^0 \Psi^{\zeta,\zeta-k}(X).
$$ 
The iterative problem to be solve is now
\begin{equation}
P(\zeta)M_{1,k}(\zeta)-Q_{1,k}(\zeta)=x'Q_{1,k+1}
\qquad Q_{1,k}\in \,^0 \Psi^{(\zeta+1,\zeta'),\zeta-k}(X).
\end{equation}
The argument to solve this equation is the same used before. Using the normal operator and proposition \ref{b42poly} one finds $M_{1,k}\in\,^0 \Psi^{\zeta,\zeta-k}.$ Again  the indicial operator cancels the leading order term and  we obtained the even better remainder
\begin{equation}
P(\zeta)\cdot M_1(\zeta) -Q_1(\zeta)=Q_2(\zeta)\in R^\infty \cdot \,^0\Psi^{(\zeta+1,\zeta'),\zeta}(X).
\end{equation}
This concludes the second stage of the parametrix.

The last stage is to remove the Taylor series from the right hand side of the previous equation. We want to find $M_2(\zeta)\in\,^0\Psi^{\zeta,\zeta}(X)$ so that
\begin{equation}\label{lastpara}
P(\zeta)\cdot M_2(\zeta)-Q_2=Q_3\in R^\infty \cdot\,^0\Psi^{\infty,\zeta}(X).
\end{equation}
This involves solving away the term  $R^\infty \cdot \,^0 \Psi^{(\zeta+1;\zeta'),\zeta}(X).$ Since the kernel of $Q_2$ vanishes to infinite order at the front face of $X\times_0 X$ it can be projected to $X\times X$ to a function in
\begin{equation}
(x)^{\zeta+1}(\ln x)^{\zeta'}(x')^{\zeta} \mca_{0,0}(X).
\end{equation}
To solve \eqref{lastpara} modulo such an error, on the right hand side the argument is the one used for proving lemma \ref{inductparametrix}. The parametrix follows with $M=M_0-M_1+M_2.$
\end{proof}

The operator  $(I-F_\lambda)$ is invertible by analytic
Fredholm theory since $F_\lambda$ is a compact operator in weighted $L^2$ spaces, and the argument of the second paragraph in the proof of theorem 7.1 on page 301 of \cite{MM} ensures that $(I-F_\lambda)$ is invertible for $\Re \lambda$ sufficiently large, where $F_\lambda$ may need to be modified by adding an elliptic operator. 
Thus we can decompose the resolvent as the pull-back using the blow-down
map b (that is $^0\tilde\Psi^m,\,^0\Psi^{\zeta,\zeta}$),
and its residual class $(\Psi^{\zeta,\zeta}),$
for details we refer to \cite{Borth}, this completes the proof of theorem \ref{resolvent},

\subsection{\bf The Poisson operator and the scattering matrix}

The proof of the existence of the Poisson operator and the scattering matrix
follow the
same as in \cite{Borth}, for that we need to analyze the Eisenstein function
\begin{equation*}
E_\zeta=(x')^{-\zeta}\mcr_\zeta \mid_{x'=0}.
\end{equation*}
We use the decomposition as in theorem \ref{resolvent}, $$\mcr_\zeta=\mcr_{1_\zeta}+\mcr_{2_\zeta}$$ with
$$\mcr_{1_\zeta}\in\, ^0\tilde\Psi^{-2}$$ and $$\mcr_{2_\zeta}\in\,^0\tilde\Psi^{\zeta,\zeta}
+\tilde\Psi^{\zeta,\zeta}.$$ The restriction $$(x')^{-\zeta}\mcr_{1_\zeta} \mid_{x'=0}$$ vanishes, thus we only need to look at the restriction
\begin{equation*}
E_\zeta=(x')^{-\zeta}\mcr_{2_\zeta} \mid_{x'=0}.
\end{equation*}
Denoting by $E_\zeta$ also the Schwartz kernel of the Eisenstein function, we have
\begin{equation}\label{poissonkerdec}
E_\zeta=E_{1_\zeta}+E_{2_\zeta},
\end{equation}
with
\begin{gather*}
b^*E_{1_\lambda}\in \mca_{\zeta,-\zeta}(X\times_0\p X), \quad\mbox{and}\quad E_{2_\zeta}\in\mca_\zeta(X\times\p X).
\end{gather*}

Let
$$
\mcr_\zeta\in \mca_{\zeta,\zeta,\infty} (X\times_0 X)+
\mca_{\zeta,\zeta} (X\times X)
$$
also denote the Schwartz kernel of the resolvent. Then the Poisson kernel is equal to
\begin{equation*}
E_\zeta=C(\zeta)(x')^{-\zeta}\mcr_\zeta \mid_{x'=0}.
\end{equation*}
Since this depends on the restriction to $x'=0$ only, to prove theorem \ref{defsc} we need prove a decomposition of the Poisson kernel
\begin{equation}\label{pois-scatt}
E_\zeta f=\mca_\zeta(X)+\mca_{n-\zeta}(X).
\end{equation}

The following theorem was proved in \cite{Borth} for those settings, the proof for this settings follows his proof verbatim

\begin{thm}
For the Schwartz kernel of the Poisson operator:
\begin{equation*}
E_\zeta f=\int_{\p X} E_\zeta(x,z')f(y')d\mu_{\mid \p X}(y')
\end{equation*}
and f$\in C^{\infty}(\p X)$, we have:
\begin{equation*}
E_\zeta f=\mca_\zeta(X)+\mca_{n-\zeta}(X).
\end{equation*}
\end{thm}

Theorem \ref{new1} follows, and shows that $E_\zeta f$ can be thought as having  $x^{\zeta}a_0$ and $x^{n-\zeta}b_0$ as their leading terms, so that $a_0$ and $b_0$ are the leading coefficients, holomorphic on $\zeta$ for $\zeta\in\mc\backslash(\Gamma\cup D).$

The explicit formula for the pull back $b^*$ of the scattering matrix $S(\zeta)$ was calculated in \cite{jsb0}, and as the Eisenstein operator depends only on the restriction to the right face. The scattering matrix is defined as in \cite{Borth} by
$$
S(\zeta) f=\frac{1}{M_\zeta} x^{-\zeta}E_{\zeta}f\mid_{\p X},
$$
where $M_\zeta$ depends on $n$ but not on g, and $S(\zeta)$ is defined for the values of $\zeta$ for which $M_\zeta$ does not vanish; we remark that it was proven in \cite{jsb0} that the set of zeros of $M_\zeta$ is a discrete set and the same proof holds for our case. As in \cite{jsb0} we have
\begin{equation}\label{Mscatt}
 b_\p^*S(\zeta)=\frac{1}{M_\zeta}b^*(x^{-\zeta}(x')^{-\zeta}\mcr_\zeta)\mid_{R\cap L}=\frac{1}{M_\zeta}b^*(x^{-\zeta}E_\zeta)\mid_{R}.
\end{equation}

The principal symbol of the scattering matrix is
\begin{equation*}
S(\zeta)=2^{n-2\zeta}\frac{\Gamma(n/2-\zeta)}{\Gamma(\zeta-n/2)}
|\eta|^{2\zeta-n}_{h_0}
\end{equation*}
for $\zeta \in \mc \backslash \Gamma$.

\section{\bf Scattering on asymptotically hyperbolic manifolds with polyhomogeneous metric }

We study the poles of the resolvent on asymptotically hyperbolic manifolds with polyhomogeneous metric. As an application we get some essential singularities of it. We assume $(X,g)$ is a asymptotically hyperbolic manifolds with metric $g$ of the form \ref{hpoly} with
\begin{multline}\label{respoly}
h=h_0+h_2 x^2+ (\hbox{even powers}) + h_{k,m_k} x^k (\ln x)^{m_k} +\cdots\\
+ h_{k,1} x^k (\ln x) +h_k x^k + O(x^{k+1}(\ln x)^{m_{k+1}})
\end{multline}
with $m_k>0.$ An example of such a manifold is an $n+1$ odd dimensional conformally compact Einstein manifold, for which $k=n,$ and $m_k=1.$ However Einstein manifolds satisfy the addition hypothesis $\Ric g=-n g,$ which makes some terms involve in the iterative construction we discuss next vanish. This is actually the reason we can to prove the inverse theorem using the MSO in section \ref{seceinst}.

We recall the identity
which holds in a collar neighborhood of the boundary $\p X,$ namely 
$(\Delta_g-\zeta(n-\zeta)) x^{n-\zeta}=x^{n-\zeta} \mcd_\zeta$ with
$$
\mcd_\zeta=-(x\p_x)^2+(2\zeta-n-\frac{x}{2}\Tr_h(\p_x h))x\p_x-\frac{(n-\zeta)x}{2}\Tr_h(\p_x h)+x^2\Delta_h.
$$
For $f\in C^\infty(\p X),$ $j,i\in \mn_0$ we have 
\begin{multline}
\mcd_\zeta(f x^j (\ln x)^i )=j(2\zeta-n-j)f x^j(\ln x)^i\\+
i(2\zeta-n-2j)f x^j(\ln x)^{i-1}-i(i-1)f x^j(\ln x)^{i-2}\\
+x^j(\ln x)^i G(n-j)f-\frac{i}{2}\Tr_h(\p_x h)fx^{j+1}(\ln x)^{i-1}\\-\frac{n-\zeta}{2}\Tr_h(\p_x h)fx^{j+1}(\ln x)^{i}
\end{multline}
with
\begin{equation}
(G(z)f)(x,y)=x^2\Delta_h f-\frac{(n-z)x}{2}\Tr_h(\p_x h)f.
\end{equation}

We can hence carry the iterative construction of \cite{GZ} for a general polyhomogeneous metric as long as $\Re \zeta\geq n/2,$ $\zeta\notin n/2+\mn_0/2,$  $\zeta(n-\zeta)\notin\sigma_{pp}(\Delta_g).$ The Poisson operator
\begin{multline}
\Phi(\zeta) f=f x^{n-\zeta}+p_{1,m_1}(\zeta)fx^{n-\zeta} x (\ln x)^{m_1}+\cdots+p_{1,1}(\zeta) fx^{n-\zeta} x (\ln x)+p_1(\zeta) fx^{n-\zeta}x\\+\cdots
+p_{j,m_j}(\zeta)fx^{n-\zeta} x^j (\ln x)^{m_j}+\cdots+p_j(\zeta) fx^{n-\zeta} x^j+\cdots 
\end{multline}
can be defined by the iterative rule $F_0=f,$ $F_{1,m_1}=F_0+x(\ln x)^{m_1}\frac{\left[x^{-1}(\ln x)^{-m_1}\Delta_g F_0 \right]\mid_{x=0}}{2\zeta-1-n},$ and  
\begin{gather}\begin{gathered}
F_{j,i}=F_{j,i+1}+x^j(\ln x)^{i}\frac{\left[x^{-j}(\ln x)^{-i}\Delta_g F_{j,i+1} \right]\mid_{x=0}}{j(2\zeta-j-n)},\quad\mbox{for}\quad m_j>i\geq 0\\
F_{j,m_j}=F_{j-1}+x^j\frac{\left[x^{-j}\Delta_g F_{j-1} \right]\mid_{x=0}}{j(2\zeta-j-n)}.
\end{gathered}\end{gather}

We now turn our attention to a metric of the form \ref{respoly}. For exceptional values of $\zeta,$ 
when $j=2\zeta-n:$
\begin{gather}
F_{j,i}=F_{j,i+1}+x^j(\ln x)^{i}\frac{\left[x^{-j}(\ln x)^{-i}\Delta_g F_{j,i+1} \right]\mid_{x=0}}{i(2\zeta-2j-n)}.\quad\mbox{for}\quad m_j>i\geq 0.
\end{gather}
The Poisson operator can also be constructed as a limit of when $\zeta$ approaches such values. By the iterative construction which involves dividing by $2\zeta-n-l$ we have that $p_{l,m_l}$ has at most a simple pole, we write 
$$
p_{l,m_l}(\zeta)=\frac{\tilde p_{l,m_l}(\zeta)}{2\zeta-n-l}+\dot p_{l,m_l}(\zeta),
$$
where $\dot p_{l,m_l}(\zeta)$ has no pole at $\zeta=n/2+l/2;,$ subsequently 
$$
p_{l,m_l-k}(\zeta)=\frac{\tilde p_{l,m_l-k}(\zeta)}{(2\zeta-n-l)^{k+1}}+\dot p_{l,m_l-k}(\zeta),
$$
with  $\dot p_{l,m_l-k}(\zeta)$ having lower order pole at $\zeta=n/2+l/2.$
First we study
$$
\Phi_l(\zeta)=\Phi(\zeta)-\Phi(n-\zeta)p_l(\zeta)
$$
and look at the terms with $x^l,$ which are
\begin{multline}\label{pl-pl}
p_{l,m_l}(\zeta)x^{n-\zeta}x^l(\ln x)^{m_l}+p_{l,m_l-1}(\zeta)x^{n-\zeta}x^l(\ln x)^{m_l-1}+\cdots\\
+p_{l,1}(\zeta)x^{n-\zeta}x^l(\ln x)+p_l(\zeta)x^{n-\zeta}x^l-p_l(\zeta)x^\zeta.
\end{multline}
We take the limit when $\zeta$ approaches $n/2-l/2,$ observe that 
$$
\tilde p_l(\zeta)x^{n-\zeta}x^l-\tilde p_l(\zeta)x^\zeta=
((2\zeta-n-l)\tilde p_l(\zeta))x^\zeta(\frac{x^{n-2\zeta+l}-1}{2\zeta-n-l})
$$
and notice that the pole of $p_{l,m_l}$ will appear as a pole of one order higher  of $p_{i,m_l-1}$ only from the factor $(m_l)(2\zeta-n-2l) p_{l,m_l}x^l(\ln x)^{m_l-1}f,$ in $D_\zeta(x^l(\ln x)^{m_l})f.$ This is the observation which allows the compensation of those poles by the appearance of a logarithmic terms of one order higher at those poles. We continue to consider
\begin{equation}
\tilde p_{l,1}(\zeta)x^{n-\zeta}x^l(\ln x)+\tilde p_l(\zeta)x^{n-\zeta}x^l-p_l(\zeta)x^\zeta
=(2\zeta-n-l)x^\zeta \tilde p_{1,1}\left( \frac{x^{n-2\zeta+l}\ln x-2\frac{x^{n-2\zeta+l}-1}{2\zeta-n-l}}{2\zeta-n-l} \right),
\end{equation}
where the two in front of  $\frac{x^{n-2\zeta+l}-1}{2\zeta-n-l}$ appears from the two in front of  $(2\zeta-n-2l)x^l(\ln x)f$ when taking $D_\zeta(x^l(\ln x)^2 f).$ The process can be continued to get that \eqref{pl-pl} equals
\begin{multline}
(2\zeta-n-l) \tilde p_{1,m_l}x^\zeta\times\\
\times\left( \frac{x^{n-2\zeta+l}(\ln x)^{m_l}-(m_l+1)\xi\left(\frac{x^{n-2\zeta+l}(\ln x)^{m_l-1}-(m_l)\xi\frac{\iddots^{x^{n-2\zeta+l}\ln x-2\xi\cdot\frac{ x^{n-2\zeta+l}-1}{2\zeta-n-l}}}{2\zeta-n-l}}{2\zeta-n-l}\right)}{2\zeta-n-l} \right),
\end{multline}
with $\xi=2\zeta-n-2l.$
This is actually the expansion of $\xi^{m_k}(m_l+1)!\cdot x^{n-2\zeta+l},$ and to find the limit we need to look at
the term with $(n-2\zeta+l)^{m_l+1}$ in the expansion near
$t=n-2\zeta+l=0$ of $x^t,$ and multiply by $l^{m_k}(m_l+1)!,$ which is $l^{m_k}(m_l+1)!\frac{(\ln x)^{m_l+1}}{(m_l+1)!}$ thus taking the limit as $\zeta$
approaches $n/2+l/2$ we get 
that \eqref{pl-pl} equals
\begin{equation}
-2l^{m_k}[(\zeta-n/2+l/2 ) p_{l,m_l}]|_{\zeta-n/2+l/2 }x^l(\ln x)^{m_l+1}.
\end{equation}
Integration by parts as in \cite{GZ}, and the previous argument gives

\begin{prop}\label{uniquepoisson}
Let $(X,g)$ be a conformally compact manifold with polyhomogeneous metric $g.$ Then for $\Re\zeta\geq n/2,$ $\zeta\notin n/2+\mn_0/2,$ $\zeta(n-\zeta)\notin\sigma_{pp}(\Delta_g),$ $\zeta\neq n/2,$ there exists a unique linear operator $\mcp(\zeta)$ such that
\begin{equation}
\begin{gathered}
\mcp(\zeta):C^\infty (\p X, |N^* \p X|^{n-\zeta})\rightarrow \mca_0 ( \intx)\\
(\Delta_g-\zeta(n-\zeta) )\mcp(\zeta)=0\\
\mcp(\zeta) f=x^{n-\zeta}F+x^\zeta G\qquad\mbox{if}\qquad\zeta\notin n/2+\mn_0/2\\
\mcp(\zeta) f=x^{n/2-l/2}F+x^{n/2+l/2} (\ln x)^{m_l+1}G\qquad\mbox{if}\qquad \zeta=n/2+l/2, \quad l\in\mn 
\end{gathered}
\end{equation}
where $F, G\in\mca_0(X).$ When $\zeta=n/2+l/2,$ $G\mid_{\p X}=-2[(\zeta-n/2+l/2 ) p_{l,m_l}]|_{\zeta-n/2+l/2 }.$
\end{prop} 

The following theorem can be proven as in \cite{GZ}:
\begin{thm}
The scattering matrix $\mcs_\zeta$ has a pole of order $m_l$ at $\zeta=n/2+l/2$ and
\begin{equation}
[(\zeta-n/2+l/2 )^{m_l} \mcs_\zeta]|_{\zeta=n/2+l/2 }=\Pi_{n/2+l/2}-2l^{m_k}[(\zeta-n/2+l/2 )^{m_l} p_{l,m_l}]|_{\zeta=n/2+l/2 }
\end{equation}
with $\Pi_{n/2+l/2}$ a finite rank operator with Schwartz kernel given by
$$
\pi_{n/2+l/2}=
(l^{m_k}(xx')^{n/2+l/2}[(\zeta-n/2+l/2 ) \mcr_\zeta]|_{\zeta=n/2+l/2 }|{\p X\times \p X}
$$
\end{thm}

We explicitly calculate what $[(\zeta-n/2+l/2 )^{m_l} p_{l,m_l}]|_{\zeta-n/2+l/2 }$ is, for the case of a metric of the form \eqref{respoly} and $l=k+1.$ For that we carry the iterative expansion at $\zeta=n/2+l/2,$ and we take $l=k+1.$ For the first terms we can proceed as before and the $j=2\zeta-n$(when the denominator vanishes) term comes from $D_\zeta(F_{j-1})$ and is $k\frac{n-\zeta}{2}x^k(\ln x)^{m_k}\Tr(h_0h_{k,m_k})$ for $j-1=k,$ or $l=k+1.$ We have proven

\begin{prop}
Let $(X,g)$ be asymptotically hyperbolic manifold with metric $g$ of the form \eqref{hmet} with $h$ of the form \eqref{respoly}. Then for $l=k+1$
$$
[(\zeta-n/2+l/2 )^{m_l} p_{l,m_l}]|_{\zeta-n/2+l/2 }=l^{m_l}k\frac{(n-k)/2-1/2}{4}\Tr(h_0h_{k,m_k}).
$$  
\end{prop}

Now we use the factorization of the scattering matrix (see \cite{Guil})
\begin{equation}
\tilde S(\zeta):=c(n-\zeta) \Lambda^{-\zeta+n/2}S(\zeta)\Lambda^{-\zeta+n/2}
\end{equation} 
with
\begin{equation}
\Lambda=(1+\Delta_{h_0})^{1/2},\quad c(\zeta)=2^{n-2\zeta}\frac{\Gamma(n/2-\zeta)}{\Gamma(\zeta-n/2)}.
\end{equation}
Since the principal symbol of the scattering matrix is given by
$$
\sigma_p(S(\zeta))=c(\zeta)\sigma_p(\Lambda^{2\zeta-n}),
$$
 we can express $\tilde S(\zeta)=1+K(\zeta)$ where $K(\zeta)$ is a compact operator for $\zeta\in \mc\backslash((n/2-\mn_0/2)\cup\mz/2).$

We also use the following lemma whose prove follows closely the one of lemma 4.2 of \cite{Guil}

\begin{lemma}
Let $B$ be a Banach space of infinite dimension, let $\zeta_0\in\mc,$ $m\in\mn,$ and $U$ be a neighborhood of $\zeta_0.$ Let $M(\zeta)\in Hol(U\backslash\{\zeta_0\},L(B))$ be a meromorphic family of bounded operators in U satisfying
$$
M(\zeta)=1+\frac{K_{m}}{(\zeta-\zeta_0)^m}+K(\zeta),\qquad ,
$$ 
with $K_m$ compact and $K(\zeta)$ compact having a pole of order $m-1$ and $\dim \ker K_m<\infty.$ If there exists $z\in U$ such that $M(z)$ is invertible, then $M(\zeta)$ is invertible for almost every $\zeta\in U$ with inverse $M^{-1}(\zeta)$ finite-meromorphic (in the sense of \cite{Guil}) and $\zeta_0$ is an essential singularity of $M^{-1}(\zeta).$
\end{lemma}

We obtain corollary \ref{corres} from the application of the previous results as in \cite{Guil}. 

\section{\bf The inverse problem}\label{secinvprob}

We analyze the relation between the symbol of the scattering matrix and the metric, for that we fix a product structure for which  
\begin{gather}\label{V&g}
  \begin{gathered}
    g_j=\frac{dx^2}{x^2}+\frac{h_j(x,y,dy)}{x^2},\qquad i=1,2.
  \end{gathered}
\end{gather}
Furthermore we assume the metrics $g_1,$ $g_2,$ are related by
$$
h_2(x,y,dy)=h_1(x,y,dy)+x^k(\ln x)^mL(x,y,dy)+O(x^k(\ln x)^{m-1}),
$$
where
$$
L( x,y,dy)=\sum_{i,j} L_{ij}(x,y)dy_idy_j.
$$
Let $P_1$ and $P_2$ be the operators
\begin{gather*}\begin{gathered}
 P_1=\Delta_{g_1}-\zeta_1(n-\zeta_1),\\
P_2=\Delta_{g_2}-\zeta_2(n-\zeta_2),
\end{gathered}\end{gather*}
and  $S_2,$ and $S_1$ be the scattering matrices associated to $P_1,$ and $P_2,$ respectively, and prove the following theorem, which is a central part of the computation and generalizes theorem 3.1. of \cite{jsb0},

\begin{thm}\label{p-p}
Let $g_1,$ $g_2,$ $h_1,$ $h_2,$ be as before. Then denoting by $h( x, y)$ the matrix of coefficients of the tensor $h(x,y,dy),$ we have for 
 $H=h_1(0,y)^{-1}L(x,y)h_1(0,y)^{-1},$ and $T=\Tr(h_1(0,y)^{-1}L( x,y)),$
\begin{equation*}
P_2-P_1=\\x^k(\ln x)^{m}\left(\sum_{i,j=1}^n H_{i,j}x\p_{y_i}x\p_{y_j}+\frac{k(k-n)}{4}h^{-1}_{1_{i,j}}(0,y)T \right)+x^{k}(\ln x)^{m-1}R.
\end{equation*}

\end{thm}
\begin{proof}

We want to look at the difference $P_1-P_2$,
The metric is  $g_{00}=\frac{1}{x^2}$, and
$\delta_i=det|g|=\frac{det|h_1|}{x^{2(n+1)}}$
hence acting
on half densities
\begin{equation*}
\del^\frac{1}{4}\Delta_g(\del^{-\frac{1}{4}}f)=
\sum_{i,j=0}^n \del^{-\frac{1}{4}} \p_{z_i}(g^{ij}(f(\p_{z_j}\del^\frac{1}{4})-
\del^\frac{1}{4}(\p_{z_j}f))).
\end{equation*}
In local coordinates
\begin{multline}
\del^\frac{1}{4}\Delta_g(\del^{-\frac{1}{4}}f)=\del^{-1/4}\p_x  x^2 (f(\p_x\del^{1/4})-\del^{1/4}\p_x f)
+\\
+\sum_{i,j}\del^{-1/4}\p_{y_i}  x^2 (h_{ij}^{-1}(x,y))(f(\p_{y_j}\del^{1/4})-\del^{1/4}\p_{y_j}f).
\end{multline}
We analyze the difference of terms in this sum. First the difference of the terms with derivatives with respect to $"x"$
$$
D_1=\del_2^{-1/4}\p_x  x^2 (f(\p_x\del_2^{1/4})-\del_2^{1/4}\p_x f)-\del_1^{-1/4}\p_x  x^2 (f(\p_x\del_1^{1/4})-\del_1^{1/4}\p_x f),
$$
as in \cite{jsb0}\footnote{There is
a little correction to the computation in \cite{jsb0}, pointed out in
\cite{GSa}.}, we have the sum of three terms:
$$
D_1=\frac{1}{2}xf\p_x\ln(\frac{\delta_2}{\delta_1})+\frac{1}{4}x^2f\p_x^2\ln(\frac{\delta_2}{\delta_1})+
\frac{1}{16}x^2f\p_x\ln(\frac{\delta_2}{\delta_1})\p_x\ln(\delta_2\delta_1).
$$
We analyze each of these terms, to do that we recall that the quotient
$$
\frac{\delta_2}{\delta_1}=1+x^k(\ln x)^m\Tr(h_1^{-1}L)+O(x^k(\ln x)^{m-1}).
$$
Near the boundary
\begin{multline*}
\ln(\frac{\delta_2}{\delta_1})=\ln(1+x^k(\ln x)^m\Tr(h_1^{-1}L)+O(x^k(\ln x)^{m-1}))=\\x^k(\ln x)^m\Tr(h_1^{-1}L)+(x^k(\ln x)^m\Tr(h_1^{-1}L))^2O(1).
\end{multline*}
Thus we have:
\begin{gather*}
\p_x\ln(\frac{\delta_2}{\delta_1})=kx^{k-1}(\ln x)^m\Tr(h_1^{-1}L)+O(x^{k-1}(\ln x)^{m-1}),
\end{gather*}
and
\begin{multline*}
\p_x^2\ln(\frac{\delta_2}{\delta_1})=k(k-1)x^{k-2}(\ln x)^m\Tr(h_1^{-1}L)+kmx^{k-2}(\ln x)^{m-1}\Tr(h_1^{-1}L)+\\
+m(k-1)x^{k-2}(\ln x)^{m-1}(\Tr(h_1^{-1}L)+m(m-1)x^{k-2}(\ln x)^{m-2}\Tr(h_1^{-1}L))+O(x^{k-1}(\ln x)^{m-1})=\\
k(k-1)x^{k-2}(\ln x)^m\Tr(h_1^{-1}L)+O(x^{k-2}(\ln x)^{m-1}).
\end{multline*}
Also
\begin{multline*}
\p_x\ln(\frac{\delta_2}{\delta_1})\p_x\ln(\delta_2\delta_1)=\\
(k x^{k-1}(\ln x)^m\Tr(h_1^{-1}L)+O(x^{k-1}(\ln x)^{m-1}))
\cdot(-4(n+1)x^{-1}+O(x^{-1}(\ln x)^{-1}))=\\
-4(n+1)k x^{k-2}(\ln x)^m\Tr(h_1^{-1}L)+O(x^{k-2}(\ln x)^{m-1}).
\end{multline*}
Substituting the later equations into $D_1$ we get
\begin{multline*}
D_1=\frac{h^{-1}_{1_{i,j}}(0,y)}{2}f(kx^{k}(\ln x)^m\Tr(h_1^{-1}L)+O(x^{k}(\ln x)^{m-1}))+\\
+\frac{1}{4}h^{-1}_{1_{i,j}}(0,y)f(k(k-1)x^{k}(\ln x)^m\Tr(h_1^{-1}L)+O(x^{k}(\ln x)^{m-1}))+\\
-\frac{(n+1)}{4}h^{-1}_{1_{i,j}}(0,y)f(k x^{k}(\ln x)^m\Tr(h_1^{-1}L)+O(x^k(\ln x)^{m-1}))=\\
\frac{h^{-1}_{1_{i,j}}(0,y)}{4}f(k(n-k)x^{k}(\ln x)^m\Tr(h_1^{-1}L)+O(x^{k}(\ln x)^{m-1})).
\end{multline*}
For the ones with derivatives with respect to $"y"$ the calculations are similar,
$$
D_{ij}=\del_2^{-1/4}\p_{y_i}  x^2 (h_{2_{ij}}^{-1}(0,y))(f(\p_{y_j}\del_2^{1/4})-\del_2^{1/4}\p_{y_j}f)-\del_1^{-1/4}\p_{y_i}  x^2 (h_{1_{ij}}^{-1}(0,y))(f(\p_{y_j}\del_1^{1/4})-\del_1^{1/4}\p_{y_j}f).
$$
For the rest of the terms in the difference, writing $h_{2_{i,j}}^{-1}(x,y)=h_{1_{i,j}}^{-1}(x,y)+ x^k (\ln x)^{m}[h_1^{-1}L h_1^{-1}]_{i,j}+O(x^{k}(\ln x)^{m-1}),$ we have
\begin{multline*}
D_{ij}=(\p_{y_i} x^2h^{-1}_{1_{i,j}})f(\del_2^{-1/4}\p_{y_j}\del_2^{1/4}-\del_1^{-1/4}\p_{y_j}\del_1^{1/4})\\
+x^2h^{-1}_{1_{i,j}}f(\del_2^{-1/4}\p_{y_i}\p_{y_j}\del_2^{1/4}-\del_1^{-1/4}\p_{y_i}\p_{y_j}\del_1^{1/4})\\
+\del_1^{-1/4}\p_{y_i}[ x^{k+2}(\ln x)^{m}[h_1^{-1}L h_1^{-1}]_{i,j}(f(\p_{y_j}\del^{1/4})-\del^{1/4}(\p_{y_j}f))],
\end{multline*}
and the only one that will contribute to the higher order sum is the last one. This concludes the proof.
\end{proof}
Now we use the theorem to compute the leading singularity for the difference of scattering matrices $S_2(\zeta)-S_1(\zeta).$ As in theorem \ref{p-p} let
\begin{equation*}
P_2-P_1=x^k(\ln x)^{m}E+x^k(\ln x)^{m-1}R,
\end{equation*}
with
\begin{equation*}
E=\sum_{i,j=1}^n H_{i,j}x\p_{y_i}x\p_{y_j}+\frac{k(k-n)}{4}h^{-1}_{1_{i,j}}(0,y)T. 
\end{equation*}
To higher order
\begin{gather*}
P_2(\mcr_1-\mcr_2)=(P_2-P_1)\mcr_1=x^k(\ln x)^{m} E \mcr_1,
\end{gather*}
looking for $\mcr_2$ as a perturbation of $\mcr_1$ leads to finding $F$ so that:
\begin{gather}\label{simpleproof}
P_2(F)=x^k (\ln x)^{m}E \mcr_1.
\end{gather}

We continue to state the theorem 2.1  of \cite{jsb0} for this setting, this gives
information on the pull-back of the difference of scattering matrices, and
hence on the leading singularity of  this difference.

\begin{thm}\label{leadsing}
Let $B_\zeta$ be the Schwartz kernel of $S_2(\zeta)-S_1(\zeta).$ The leading singularity
of $B_\zeta$ is given by
\begin{equation*}
\frac{C(\zeta)}{M(\zeta)}\left[
T_1(k,\zeta)\sum_{i,j=1}^\infty H_{ij}(y)\p_{Y_i}\p_{Y_j}(\ln|Y|)^m |Y|^{2 \zeta-n-k-2}
-T_2(k,\zeta)(
\frac{k}{4}(k-n)T(y))(\ln|Y|)^m|Y|^{2\zeta-n-k}\right]
\end{equation*}
times a non-vanishing $C^2$ half-density.
\end{thm}

\begin{proof}
We decompose the Poisson kernel 
\begin{equation}%\label{poissonkerdec}
E_\zeta=E_{1_\zeta}+E_{2_\zeta},
\end{equation}
with
\begin{gather*}
b^*E_{1_\zeta}\in \mca_{\zeta,-\zeta}(X\times_0\p X), \quad\mbox{and}\quad E_{2_\zeta}\in\mca_\zeta(X\times\p X).
\end{gather*} 
By  \eqref{Mscatt} we have $b_\p^*S(\zeta)=\frac{1}{M_\zeta}b^*(x^{-\zeta}E_\zeta)\mid_{R}.$ Thus the only term contributing to the difference of scattering matrices will be $E_{1_\zeta}$, i.e.  $b_\p^*(S_2(\zeta)-S_1(\zeta))=\frac{1}{M_\zeta}b^*((x^{-\zeta}(E_{2_\zeta}-E_{1_\zeta})\mid_{R}.$

To find this difference we start from \ref{simpleproof}
\begin{equation}\label{pf}
P_2(F)=x^k (\ln x)^{m}E \mcr_1= x^k(\ln\rho+\ln R)^mE\mcr_1
=\sum_{l=0}^m B(l,m)x^k (\ln\rho)^{m-l}(\ln R)^l E \mcr_1.
\end{equation}
We look for $F$ of the form $F=(x')^k(\ln x')^mF_1,$ we put this into \eqref{pf}, using the $P_2$ commutes with $x'$ we get
\begin{equation}\label{pf2}
(\ln x')^mP_2(F_1)=\sum_{l=0}^m B(l,m)s^k (\ln\rho)^{m-l}(\ln R)^l E \mcr_1.
\end{equation}
Using that 
$$
(\ln x')^m=(\ln \rho'+\ln R)^m=\sum_{l=0}^m B(l,m) (\ln\rho')^{m-l}(\ln R)^l,
$$
we get
\begin{equation}
\sum_{l=0}^m B(l,m) (\ln\rho')^{m-l}(\ln R)^l P_2(F_1)=
\sum_{l=0}^m B(l,m)s^k (\ln\rho)^{m-l}(\ln R)^l E \mcr_1.
\end{equation}
We look for the leading singularity on the intersection of the left and right faces ($\rho=\rho'=0$), at this intersection we set the terms with same power of $\ln R$ to be equal and get that  modulo lower order terms $F=B(m,m)(\ln |y-y'|)^m F_1$ with $F_1$ satisfying
\begin{equation}\label{pf4}
P_2(F_1)= s^k E \mcr_1.
\end{equation}
Thus the calculations necessary to obtain the theorem are now those of the proof of lemma 5.1 of \cite{jsb0}.
\end{proof}
Having this lemma, we obtain theorem \ref{invpoly} since we can use it in the same way it is used in \cite{jsb0}.

\section{\bf The inverse problem for Einstein manifolds with odd metric}\label{seceinst}

In this section we  prove  theorem \ref{inveinstein}. For $n+1$ even an inverse result was proved in \cite{GSa}. We assume that $n+1$ is odd for the rest of the section.

\subsection{\bf The Scattering Operator}

  Let $X$ be an $n+1$ dimensional conformally compact Einstein manifold as defined in the introduction. We begin by characterizing the Scattering map. To do that we write the asymptotic series expansion of the solution of $u$ given in theorem \ref{new1}; we recall the Laplacian   
$$
\Delta_g=-(x\p_x)^2+(n-\frac{x}{2}\Tr_h(\p_x h))x\p_x+x^2\Delta_h.
$$
We denote by $\sigma_{pp}(\Delta_g)$ the pure point spectrum of $\Delta_g.$
A first naive try would be to get an asymptotic expansion of the form 
$$
u(x,y)=\sum_{j=0}^\infty f_j x^j. 
$$
By putting together the terms in this Laplacian involving the metric $h$ and the rest, and applying the Laplacian to $f_j x^j,$ a  recursive relation
$$
F_j=\sum_{k=0}^{j}x^kf_k(y),\qquad F_0=f_0=f,\qquad F_j=F_{j-1}+x^j\frac{[x^{-j}(\Delta_g F_{j-1})]\mid_{x=0}}{j(j-n)},
$$
can be obtained. But this only works for the first $j<n$ terms, it breaks down at the nth term. For the nth term we try an polyhomogeneous expansion. We look at the effect of the Laplacian on the logarithmic term $p_n(y)x^n\ln x:$
\begin{equation}\label{efflog}
\Delta_g(p_n(y)x^n\ln x)=-nx^np_n(y)+O(x^{n+1}\ln x).
\end{equation}
Thus setting $p_n=\frac{[x^{-n}\Delta_g(F_{n-1})]\mid_{x=0}}{n}$ the recursion relation at the nth step is $F_n=F_{n-1}+p_nx^n\ln x + f_nx^n$ with $f_n$ arbitrary (we set $f_n=0$) gives $\Delta_g(F_n)=O(x^{n+1}\ln x).$ The construction can then be continued to get$F_{\infty}=\lim_{j\rightarrow\infty} F_j,$  by Borel lemma, We can obtain an asymptotic expansion for $u=F_{\infty}-G\Delta F_{\infty}$ (via the pull-back by the flow of the gradient $\phi$) when $\zeta\rightarrow n,$ of the form
\begin{equation}\label{logasymp}
\phi^*u(x,y)\sim f(y)+\sum_{0<2j<n} x^{2j} f_{2j}(y)+p_nx^n\ln x+  \phi^*(G\Delta F_{\infty})+O(x^{n+1}\ln x).
\end{equation}
Where $\phi^*(G\Delta F_{\infty})=x^n K,$ for some polyhomogeneous operator $K$ by theorem \ref{resolvent}. We define the modified scattering operator,
$$
\tilde\mcs f=-[x^{-n}\phi^*(G\Delta F_{\infty})]\mid_{x=0}.
$$

On the other hand we follow the construction of Graham-Zworski  \cite{GZ}, and then take the limit as $\zeta\rightarrow n,$ First they construct (for $\zeta\notin n/2+\mn_0/2$) $\Phi(\zeta)$ so that,
$$
\Phi(\zeta)f=fx^{n-\zeta}+p_{1,\zeta}fx^{n-\zeta+1}+\cdots+p_{j,\zeta}fx^{n-\zeta+j}+O(x^{n-\zeta+j+1}).
$$
Then define
$$
\Phi_l(\zeta)=\Phi(\zeta)-\Phi(n-\zeta)p_{l,\zeta}.
$$
For $\zeta$ near $n$ the Poisson operator is 
\begin{gather*}
\mcp_l(\zeta)=(I-R(\zeta)(\Delta_g-\zeta(n-\zeta)))\Phi_l(\zeta).
\end{gather*}
We write the action of the Poisson operator on an initial value $f$ explicitly and then take the limit,
\begin{multline}
\mcp_l(\zeta)f=
x^{n-\zeta}(f+p_{1,\zeta}fx^{1}+\cdots+p_{n,\zeta}fx^{n}+\cdots)+\\
-x^{\zeta}(p_{l,\zeta}f+p_{1,\zeta}p_{l,\zeta}fx^{1}+\cdots)+ x^{\zeta}(S(\zeta)f+O(x\ln x)).
\end{multline}
By Proposition 3.4. of \cite{GZ}, 
$$
\mcp_l(\zeta)f=\mcp(\zeta)f=
x^{n-\zeta}(f+p_{1,\zeta}fx^{1}+\cdots+p_{n,\zeta}fx^{n}+\cdots)+x^{\zeta}(S(\zeta)f+O(x\ln x)),
$$
for $\zeta\neq n,$ $\zeta$ near $n.$ Notice that the last equation originally should look like 
$$
\mcp_l(\zeta)f=\mcp(\zeta)f=
x^{n-\zeta}(f+p_{1,\zeta}fx^{1}+\cdots+p'_{n,\zeta}fx^{n}(\ln x)+p_{n,\zeta}fx^{n}+\cdots)+x^{\zeta}(S(\zeta)f+O(x\ln x)),
$$
but $p'_{n,\zeta}=0$ since the manifold is Einstein (c.f. theorem 4.8 \cite{FG1}). 
 Using Proposition 3.6. of \cite{GZ}, since $0\notin \sigma_{pp}(\Delta_g),$ taking  the limit as $\zeta\rightarrow n,$  this has to correspond to \eqref{logasymp}. We use the Taylor expansions
\begin{gather}\begin{gathered}
x^{n-\zeta}\sim 1-(\ln x) (\zeta-n)+(\ln x)^2 (\zeta-n)^2/2+\cdots\\
x^{\zeta}\sim x^n+(\ln x) x^n(\zeta-n)+x^n(\ln x)^2 (\zeta-n)^2/2+\cdots,
\end{gathered}\end{gather}
to get that
\begin{equation}\label{resscatmat}
\Res_{\zeta=n} S(\zeta)-\Res_{\zeta=n} p_{n,\zeta}=p_n,
\end{equation}
and 
\begin{equation}\label{scatoper}
\tilde\mcs f=\frac{d[(n-\zeta)S(\zeta)]}{d\zeta}\mid_{\zeta=n}.
\end{equation}
Putting equation \eqref{resscatmat} together with Proposition 3.6 of \cite{GZ} we get
\begin{equation}
\Res_{\zeta=n} S(\zeta)=-\Res_{\zeta=n} p_{n,\zeta}.
\end{equation} 
Obtaining
\begin{equation}\label{gn}
2\Res_{\zeta=n} S(\zeta)=p_n.
\end{equation}

We state the unique continuation theorem proved in \cite{biq}, which is central for the inverse theorem. To state it we recall \cite{Grah1} that for $n+1$ even 
\begin{gather*}
h(x)=h_0(y,dy)+(\mbox{even powers}) + F_n x^n + \cdots, 
\end{gather*}
and for $n+1$ odd
\begin{gather*}
h(x)=h_0(y,dy)+(\mbox{even powers}) + h_n x^n \ln x + F_n x^n + \cdots. 
\end{gather*}
We call the coefficients $F_n$ the global terms.  

\begin{thm}\label{uniquecont}
Given two conformally compact Einstein metrics $g_1,$ $g_2,$, such that $h_1$ agree with $h_2$ at the boundary and their global term also coincide. There exists a diffeomorphism $\phi$, equal to the identity near the boundary, such that on a neighborhood of the boundary $\phi^*g_1=g_2.$  
\end{thm}
\subsection{\bf Proof of theorem \ref{inveinstein}}

We first prove the following 

\begin{lemma}

Let $X_i,$ $\p X_i,$ $g_i,$ for $i=1,2,$ be  $n+1$-dimensional Einstein manifolds. Let $S_i$ for $i=1,2,$ be the corresponding scattering  matrix. Assume $\emptyset\neq\mco\subset\p X_1 \cap \p X_2$ an open set, and that $\Id: \mco\subset\p X_1\mapsto\p X_2$ is a diffeomorphism.  If 
$$
\tilde\mcs_1 f\mid_\mco=\tilde\mcs_2 f\mid_\mco
$$
for all $f\in\mc^\infty_0(\mco).$ 
Then the metrics $h_i$ and $h_2$ with asymptotic expansion given by
\begin{gather*}
h_i(x)=h_{i0}(y,dy)+ (\mbox{even powers}) + h_{in} x^n \ln x + F_{in} x^n +\cdots\qquad i=1,2; 
\end{gather*}
satisfy $ h_{10}\mid_\mco=h_{20}\mid_\mco$ and $F_{1n}\mid_\mco=F_{2n}\mid_\mco.$
\end{lemma} 
\begin{proof}
The proof is analog to the proof for when $n+1$ is even. 
If
$$
\tilde\mcs_1 f\mid_\mco=\tilde\mcs_2 f\mid_\mco
$$
the principal symbol of $\tilde\mcs$ is given by
$$
2^{-n+1}\frac{(n-\zeta) \Gamma (\frac{n}{2}-\zeta)}{\Gamma (\zeta-\frac{n}{2})}\mid_{\zeta=n}|\eta|^{n}_{h_i(0)}(\ln |\eta|_{h_i(0)}),  
$$
which implies that $h_1(0)=h_2(0).$
We can also compute $(n-\zeta)(S_1(\zeta)-S_2(\zeta))$ which has principal symbol
\begin{multline*}
(n-\zeta)\frac{C(\zeta)}{M(\zeta)}\times\\
\times\left[
T_1(k,\zeta)\sum_{i,j=1}^\infty H_{ij}(y)\p_{Y_i}\p_{Y_j}(\ln|Y|)^m |Y|^{2 \zeta-n-k-2}
-T_2(k,\zeta)(
\frac{k}{4}(k-n)T(y))(\ln|Y|)^m|Y|^{2\zeta-n-k}\right].
\end{multline*}
Taking derivative with respect to $\zeta$ and evaluating at $\zeta=n$ we get 
\begin{multline*}
\left((n-\zeta)\frac{C(\zeta)}{M(\zeta)}T_1(k,\zeta)\right)\mid_{\zeta=n}
2\sum_{i,j=1}^\infty H_{ij}(y)\p_{Y_i}\p_{Y_j}(\ln|Y|)^{m+1} |Y|^{n-k-2}\\
-\left((n-\zeta)\frac{C(\zeta)}{M(\zeta)}T_2(k,\zeta)\right)\mid_{\zeta=n}
2\frac{k}{4}(k-n)T(y))(\ln|Y|)^{m+1}|Y|^{n-k}
\end{multline*}
+
\begin{multline*}
\left((n-\zeta)\frac{C(\zeta)}{M(\zeta)}T_1(k,\zeta)\right)'\mid_{\zeta=n}
\sum_{i,j=1}^\infty H_{ij}(y)\p_{Y_i}\p_{Y_j}(\ln|Y|)^m |Y|^{n-k-2}\\
-\left((n-\zeta)\frac{C(\zeta)}{M(\zeta)}T_2(k,\zeta)\right)'\mid_{\zeta=n}
\frac{k}{4}(k-n)T(y))(\ln|Y|)^m|Y|^{n-k}.
\end{multline*}
When $k=n$ and $m=1$ we get that $L=0,$  we have then that  $h_1=h_2+x^k L_1+O(x^{n+1})$ and the same reasoning could be applied inductively to get the Lemma.
\end{proof}

 Applying theorem \ref{uniquecont} we get an isometry in a neighborhood of the boundary. To extend this isometry to the whole manifold, we apply theorem 4.1 of \cite{LTU} to the complete manifolds without boundary  $(\intx_1,g_1)$, $(\intx_2,g_2)$ just as in \cite{GSa},  to get  theorem \ref{inveinstein}.


\begin{thebibliography}{99}
\bibitem{biq} O. Biquard. {\em Continuation unique a partir de l'infini conforme pour les metriques d'Einstein}, arXiv:0708.4346.
\bibitem{Borth} D. Borthwick. {\em Scattering Theory for Conformally
Compact Metrics with Variable Curvature at Infinity.}
J. of Func. Anal. {\bf 184}, 313-376, (2001).
\bibitem{PDLS}P.T. Chru\'sciel, E. Delay, J.M. Lee, D.N. Skinner, {\em Boundary regularity of conformally
compact Einstein metrics,} J. Diff. Geom. {\bf 69} (2005), 111-136.
\bibitem{DTK}D.M. DeTurk, J.L. Kazdan, {\em Some regularity theorems in Riemannian geometry,} Ann. Sci. Ecole Norm. Sup. (4) {\bf 14} (1981), no. 3, 249-260.
\bibitem{FG} C. Fefferman, C.R. Graham {\em Conformal invariants,} in {\em The mathematical heritage of \'Elie Cartan(Lyon, 1984),} Ast\'erisque 1985, Numero Hors Serie, 95-116. 
\bibitem{FG1} C. Fefferman, C.R. Graham {\em The ambient metric,} arXiv:0710.0919.

%\bibitem{fried1} F.G. Friedlander. {\em Radiation fields and h
%scattering theory. } Math. Proc. Camb. Phil. Soc., {\bf 88}, 483-515, (1980).
%\bibitem{fried2} F.G. Friedlander. {\em Notes on the wave equation on
%asymptotically Euclidean manifolds. } J. of Func. Anal. {\bf 184},
%no. 1, 1-18, (2001).
\bibitem{Guil} C. Guillarmou {\em Meromorphic properties of the resolvent on asymptotically hyperbolic manifolds.}  Duke Math. J. {\bf 129} (2005), 1-37. MR 2153454.
\bibitem{GSa} C. Guillarmou, A. S\'a Barreto. {\em Inverse problem for
einstein manifolds.}  arXiv:0710.1136.
\bibitem{Grah1} C.R. Graham. {\em Volume and area renormalizations for conformally compact Einstein Metrics.} Rend. Circ. mat. Palermo {\bf 2} Suppl. {\bf 63}, 31-42, (2000).
\bibitem{GZ} C.R. Graham, M. Zworski, {\em Scattering matrix in conformal geometry.}  Invent. Math. {\bf 152}  {\bf No. 1}, 89-118, (2003).

%\bibitem{hava} A. Hassell and A. Vasy. {\em The spectral projections and the resolvent for scattering metrics.} J. Anal. Math. {\bf 79}
%241-298. (1999).
%\bibitem{helgason} S. Helgason. {\em The Radon Transform.} Birkhauser, 2nd
%edition, (1999).
%\bibitem{hov4} L. H\"ormander. {\em The analysis of linear partial differential
%operators IV.} Springer-Verlag, 1985.
\bibitem{jsb0} M.S. Joshi and A. S\'a Barreto. {\em  Inverse scattering on
asymptotically hyperbolic manifolds.} Acta  Math.
{\bf 184}, 41-86, (2000)
\bibitem{jsb1} M.S. Joshi and A. S\'a Barreto. {\em Recovering asymptotics
of metrics from fixed energy scattering data.} Invent. Math.
{\bf 137}, 127-143, (1999)
\bibitem{LTU} M.Lassas, M. Taylor, G. Uhlmann. {\em The Dirichlet-to-Neumann map for complete Riemannian manifolds with boundary.} Communications in Analysis and Geometry
{\bf 11}, 207-222, (2003)


%\bibitem{lp} P. Lax and R. Phillips. {\em Scattering theory. }
%Academic Press, 1989, Revised edition .
\bibitem{Mar} L.Marazzi. {\em Potential scattering on conformally compact manifolds } arXiv:0803.1298.
\bibitem{MM} R.R.Mazzeo and R. B. Melrose,  {\em Meromorphic extention of the resolvent on complete spaces with asymptotically constant negative curvature.} Journal of Functional analysis {\bf 75}, 260-310 (1987)
\bibitem{megs} R.B. Melrose. {\em Geometric scattering theory.} Cambridge
Univ. Press, 1995.
\bibitem{meeuc} R.B. Melrose. {\em Spectral and scattering theory for the
Laplacian on asymptotically Euclidean spaces. } Spectral and Scattering
 Theory.  (M. Ikawa, ed.) Marcel Dekker, 1994.
\bibitem{melcor} R.B. Melrose. {\em Calculus of conormal distributions on
manifolds with corners.} Inter. Math. Res. Not., (1992) 51-61.
%\bibitem{mzw} R.B. Melrose and M. Zworski.{\em Scattering metrics and geodesic
%flow at infinity.} Invent. Math. {\bf 124}, 389-436 (1996).
%\bibitem{ro} L. Robbiano. {\em Th\'eor\`eme d'unicit\'e adapt\'e au contr\^ole
%des solutions des probl\'emes hyperboliques.} Comm. in P.D.E {\bf 16},
%789-800, (1991).
%\bibitem{rozu} L. Robbiano and C. Zuilly. {\em Uniqueness in the Cauchy
%problem for operators with partially holomorphic coefficients.}
%Invent. Math. {\bf 131}, 493-539, (1998).
%\bibitem{sb1}  A. S\'a Barreto. {\em Radiation fields on Asymptotically Euclidean Manifolds}
%Comm. in P.D.E. {\bf 28},Nos. 9 \& 10, 1661-1673, 2003.
%\bibitem{sbhy} A. S\'a Barreto. {\em Radiation fields and inverse scattering
%on asymptotically hyperbolic manifolds.} Duke Math. Journal {\bf 129}, No 3
2005.
\bibitem{Ulh} G. Uhlmann, {\em Developments in inverse problems since Calder\'on's foundational paper.} Harmoni analisis and pertial differential equations (Chicago, IL, 1996), 295-345, Chicago lectures in Math., Univ. of Chicago Press, Chicago, IL, 1999.


%\bibitem{tat} D. Tataru. {\em Unique continuation for solutions to PDE's:
%between H\"ormander's theorem and Holmgren's theorem.} Comm. in P.D.E.
%{\bf 20}, 855-884, (1995).
%\bibitem{tat1}  D. Tataru. {\em Unique continuation for operators with
%partially analytic coefficients.} J. Math. Pures Appl. {\bf 78}, 505-521,
%1999.

\end{thebibliography}
\end{document}